\documentclass[reqno]{amsart}

\usepackage{amsmath,amssymb,amsthm}

\newtheorem{proposition}{Proposition}
\newtheorem*{thm}{Theorem}

\begin{document}

\title[]{Sums of square roots that\\ are close to an integer}

\author[]{Stefan Steinerberger}

\address{Department of Mathematics, University of Washington, Seattle, WA 98195, USA}
 \email{steinerb@uw.edu}

\keywords{Square Root Sum Problem, $\sqrt{n} \mod 1$}
\subjclass[2020]{11J71, 37A45, 65D18, 68Q17} 
\thanks{The author is partially supported by the NSF (DMS-2123224).}

\begin{abstract}  Let $k \in \mathbb{N}$ and suppose we are given $k$ integers $1 \leq a_1, \dots, a_k \leq n$. If $\sqrt{a_1} + \dots + \sqrt{a_k}$ is not an integer, how close can it be to one?  When $k=1$, the distance to the nearest integer is $\gtrsim n^{-1/2}$. Angluin-Eisenstat observed the bound $\gtrsim n^{-3/2}$ when $k=2$. We prove there is a universal $c>0$ such that, for all $k \geq 2$, there exists a $c_k > 0$ and $k$ integers in $\left\{1,2,\dots, n\right\}$ with
$$ 0 <\|\sqrt{a_1} + \dots + \sqrt{a_k} \| \leq c_k\cdot n^{-c \cdot k^{1/3}},$$
where $\| \cdot \|$ denotes the distance to the nearest integer. This is a case of the square-root sum problem in numerical analysis where the usual cancellation constructions do not apply: even for $k=3$, constructing explicit examples of integers whose square root sum is nearly an integer appears to be nontrivial.
\end{abstract}
\maketitle

\maketitle

\section{Introduction and Results}
\subsection{Introduction}
We consider a very elementary problem: if the sum of square roots of $k$ positive integers is not an integer, how close can it be to an integer? 
Using $\| \cdot \|$ to denote the distance to the nearest integer, the answer is easy when we have a single square root: the best case is being close to a square number and
$$ \sqrt{n}  \notin \mathbb{N} \quad \implies \quad \|\sqrt{n}\| \gtrsim \frac{1}{\sqrt{n}}.$$
The question becomes more interesting when $k=2$. There is a nice solution.

\begin{thm}[Angluin \& Eisenstat \cite{angluin}] Let $1 \leq a,b \leq n$. If $\sqrt{a} + \sqrt{b} \notin \mathbb{N}$, then
$$ \| \sqrt{a} + \sqrt{b}\| \gtrsim \frac{1}{n^{3/2}}.$$
\end{thm}
This has a nice proof: if $\sqrt{a} + \sqrt{b} = m + \varepsilon$ with $0 \leq |\varepsilon| \ll 1$, then squaring and rearranging leads to 
$  2m\varepsilon -  \sqrt{4ab} +  \varepsilon^2 \in \mathbb{N}.$
If $\sqrt{4ab} \in \mathbb{N}$, then this forces $2m |\varepsilon| \gtrsim 1$ implying $|\varepsilon| \gtrsim m^{-1} \gtrsim n^{-1/2}$. If not, we use $\|  \sqrt{4ab} \| \gtrsim 1/n$ to deduce $|\varepsilon| \gtrsim 1/(mn) \gtrsim n^{-3/2}$.  The bound is clearly sharp up to constants since
$$ \sqrt{a^2 + 1} +\sqrt{a^2-1} = 2a - \frac{1}{4 a^3} + o(a^{-3}).$$

 A quick inspection of the argument shows that the case $k=2$ is not so different from the case $k=1$ since squaring both sides of the equation reduces one to the other. This is no longer the case when $k \geq 3$.

\subsection{The sum of square roots problem} This is a special case of a notoriously difficult problem in Numerical Analysis and Computational Complexity Theory which has a number of different formulations. Following the formulation of Demaine, Mitchell and O'Rourke \cite[Problem 33]{demaine}, the question is, for given values $n,k \in \mathbb{N}$, to understand the smallest positive value that can be attained by
$$\left| \sum_{i=1}^{k} \sqrt{a_i} - \sum_{i=1}^{k} \sqrt{b_i} \right|$$
where the $2k$ integers $a_i, b_i$ are all chosen from $\left\{1,2,\dots, n\right\}$. 
One is interested in bounds that only depend on $n$ and $k$. The question arose in the context of comparing distances. The computational cost of deciding, by computation, whether $\sqrt{10} + \sqrt{11}$ is larger than $\sqrt{5} + \sqrt{18}$ depends on how close these the numbers can be (their difference is $2 \cdot 10^{-4}$). The problem seems to first have been mentioned by O'Rourke \cite{orourke} in 1981. A more general problem of deciding the sign of such expressions was posed in a 1976 paper of Garey, Graham and Johnson \cite{garey, grah}.
There is a big difference between the upper and the lower bounds. Burnikel, Fleischer, Mehlhorn and Schirra \cite{burn} showed that
$$  \sum_{i=1}^{k} \sqrt{a_i} \neq  \sum_{i=1}^{k} \sqrt{b_i} \implies \left| \sum_{i=1}^{k} \sqrt{a_i} - \sum_{i=1}^{k} \sqrt{b_i} \right| \gtrsim n^{- c \cdot 4^k}.$$
Explicit constructions are nowhere near as good. Note first that any such sum always lies in $[0, k \sqrt{n}]$ and that we can build approximately $\sim n^k$ different expressions:  the pigeonhole principle would then suggest that there should be examples where the gap is less than $\lesssim n^{1/2 - k}$. This was proven by Graham (see \cite{qian}) using constructions from the Prouhet–Tarry–Escott problem. Currently, the best bound is due to Qian and Wang \cite{qian} who proved that the expression is $ \lesssim n^{-2k + 3/2}$. Their argument is `local': the idea is to construct a suitable linear combination of  consecutive square roots so that a lot of cancellation occurs: they use the nice inequality
$$ \left|\sum_{i=0}^{m} \binom{m}{i} (-1)^i \sqrt{n+ i}\right| \leq \frac{(2m-3)!!}{2^m n^{m-1/2}} \qquad (\diamond).$$
Whether this upper bound is close to the truth depends on whether the best way to construct numbers close to integers is to use local cancellations in the sense of $(\diamond)$. However, there are no indications that this is actually the case: one would expect that there exist configurations of square roots which are quite different in size but then achieve very good cancellations and end up near an integer. Much stronger bounds should be true but this appears to be a very difficult problem. A nice example from the PhD thesis of Li \cite{cheng0, li} is
$$\sqrt{29} +\sqrt{1097}+\sqrt{3153} -\sqrt{226}- \sqrt{2324}-\sqrt{987} \sim 2.84 \cdot 10^{-20}.$$
Our problem is more restricted in the sense that cancellation constructions in the style of Graham and Qian-Wang \cite{qian} appear to be impossible because we are only allowed to add. Even without subtractions, there seem to be many examples that get close to integers: for example,
$$ \sqrt{3} + \sqrt{20} + \sqrt{23} = 11.000018\dots$$

\subsection{The case $k = 3$.} We return to our question: given $1 \leq a_1, \dots, a_k \leq n$, what is the smallest positive value of $\| \sqrt{a_1} + \dots + \sqrt{a_k} \|$?  The cases $k=1,2$ have nice and satisfying solutions. However, already the case $k=3$ appears to be quite nontrivial. There is a very nice argument showing that it is at least $\gtrsim n^{-7/2}$ that was communicated to us by Arturas Dubickas and Roger Heath-Brown.
\begin{proposition}
Let $1 \leq a,b,c \leq n$ be integers. Then
$$\sqrt{a} + \sqrt{b} + \sqrt{c} ~\notin \mathbb{N} \quad \implies \quad \| \sqrt{a} + \sqrt{b} + \sqrt{c} \| \gtrsim \frac{1}{n^{7/2}}.$$
\end{proposition}
\begin{proof} 
Multiplying over all possible combinations of signs $\varepsilon_1, \varepsilon_2, \varepsilon_3 \in \left\{-1,1\right\}$, one can define a polynomial 
$$ p(x) = \prod_{\varepsilon_1, \varepsilon_2, \varepsilon_3 \in \left\{-1,1\right\}} \left( x - ( \varepsilon_1 \sqrt{a} + \varepsilon_2 \sqrt{b} + \varepsilon_3 \sqrt{c})\right).$$
One sees, by direct computation, that it is a polynomial all of whose coefficients are integers. It is also a classical result, see Besicovitch \cite{bes, bor}, that if $a,b,c$ have different square-free parts, then $1, \sqrt{a}, \sqrt{b}, \sqrt{c}$ are linearly independent over $\mathbb{Q}$. This means that if $n \in \mathbb{N}$, then $p(n) \neq 0$ and thus $|p(n)| \geq 1$. Suppose now that
$ \sqrt{a} + \sqrt{b} + \sqrt{c}$ is not an integer but very close to the integer $n$.  Using $|p(n)| \geq 1$ and that the product of the other 7 factors being $\leq (3 \sqrt{n})^7$ now shows the result. 
\end{proof}

As for upper bounds, Nicholas Marshall informed us of the identity
$$ \| \sqrt{(k-1)^2 +2} + \sqrt{(k+1)^2 + 2} + \sqrt{(2k)^2 - 8} \| \sim \frac{4}{k^5}$$
which gives an infinite family of examples showing
$ \| \sqrt{a} + \sqrt{b} + \sqrt{c} \| \lesssim n^{-5/2}$. Roger Heath-Brown
reports a similar identity, there appear to be many. All these examples have
in common that each square root is a small perturbation of an integer (i.e.
of the form $\sqrt{k^2 + \mbox{small}}$). One would perhaps be inclined to
believe that there should be many `non-perturbative' examples that are $\ll n^{-5/2}?$ 
Tom\'as Oliveira e Silva informed me of the example
$$ \|\sqrt{11075} + \sqrt{27187} + \sqrt{68057}\| = 1.26 \cdot 10^{-15} \sim \frac{0.4}{68057^3}$$
which is maybe suggestive of the bound $\lesssim n^{-3}$. Since we have $\sim n^3$ expressions 
mod 1, if they behaved randomly, the scaling would be $n^{-3}$. In that sense, one of way of posing the
question is whether the sum of three roots is sufficiently `random' to forget about the algebraic
structure of each individual summand.
\vspace{2pt}

There are a number of interesting related problems, some of which might be easier. One that appears
strictly easier concerns the set
$$A = \left\{ \sqrt{a} + \sqrt{b} \mod 1:  1 \leq a,b \leq n\right\} \subset [0,1].$$
It might be interesting to have a better understanding of its gap structure. We note that $A \cap [0, c/n^{3/2}] = \emptyset$ which
follows from Angluin-Eisenstat argument. Is it true that $A$ intersects every interval $J \subset [0,1]$ with length $ |J| \gtrsim n^{-3/2}$~?
Is it true that there are `few' large gaps at scale $\sim n^{-3/2}$? 

\subsection{Main Result.} We can now state our main result.
\begin{thm} There exists $c>0$ such that for any $k \in \mathbb{N}$ fixed, there exists $c_k >0$ so that for all $n \in \mathbb{N}$, there exist $1 \leq a_1, \dots, a_k \leq n$ with
$$ 0 < \| \sqrt{a_1} + \dots + \sqrt{a_k} \| \leq c_k \cdot  n^{- c \cdot k^{1/3}}.$$
\end{thm}
\textit{Remarks.}
\begin{enumerate}
\item The proof uses both the van der Corput method (needing $k \geq 241$) and the Vinogradov method: the implicit constant $c>0$ is \textit{very} small and 
\item The argument presented in Proposition 1 generalizes to imply the lower bound $\gtrsim n^{1/2 - 2^{k-1}}$.  A polynomial rate (in $k$) appears to be out of reach. 
\item The proof shows a slightly stronger result: we show that the largest gap in the set
$\sqrt{a_1} + \dots + \sqrt{a_k} \mod 1 \subset [0,1]$ is $\leq c_k \cdot n^{- c \cdot k^{1/3}}$. In particular,
while there are summands that are guaranteed to be close to an integer, the method also guarantees
the existence of summands close to, say, $\mathbb{N} + \sqrt{2}$.
\end{enumerate}

\subsection{Exponential sum estimates} We were originally motivated by the large amount of recent work on the structure of $\sqrt{n} \mod 1$ and related problems, see \cite{aisti, browning, elbaz, elkies,  lutsko, rad, rega, rud0, rud, sinai} and references therein that establish precise results on the gap structure of this sequence. It is a very natural question whether advances from this direction could be useful for our problem. 
 A somewhat standard approach, that we will also in the proof of the Theorem, is as follows. We use a function $h:\mathbb{R}\setminus \mathbb{Z} \rightarrow \mathbb{R}$  that is essentially a characteristic function for $|x| \leq 1/s$, say
$$ h(x) = \max \left\{1 - s \cdot |x|, 0 \right\},$$
where $s \gg 1$ is a large parameter. The number of solutions $(a_1, \dots, a_k)$ where $\| \sqrt{a_1} + \dots + \sqrt{a_k} \| \leq 1/s$ is approximately given by 
$$ \sum_{a_1,a_2,\dots, a_k = 1}^{n} h(\sqrt{a_1} + \sqrt{a_2} + \dots + \sqrt{a_k} )  =  \sum_{\ell \in \mathbb{Z}} \widehat{h}(\ell) \left(\sum_{a=1}^{n} e^{2\pi i \ell \sqrt{a}}\right)^k$$
We note that this also counts the trivial solutions where all $\sqrt{a_i} \in \mathbb{N}$ and there are $\sim n^{k/2}$ of those. In order to prove the existence of nontrivial solutions close to the origin, one is required to show that
$$ \frac{n^k}{s} +  \sum_{\ell \in \mathbb{Z} \atop \ell \neq 0} \widehat{h}(\ell) \left(\sum_{a=1}^{n} e^{2\pi i \ell \sqrt{a}}\right)^k \gg n^{k/2}.$$
We will argue that for a suitable choice of $s \ll n^{k/2}$, the exponential sums are sufficiently small to be absorbed by the main term which then guarantees the existence of (many) solutions. By considering the modified function $h_y = h(x-y)$ and realizing that
$ \widehat{h_y(\ell)}$ is merely a phase-shift away from $\widehat{h}(\ell)$, a phase-shift that immediately disappears when applying the triangle inequality, we see that the argument can be equally applied to $h(x-y)$ and thus really is a statement about the number of solutions close to any arbitrary $0 \leq y \leq 1$ (not just 0).\\

The main bottleneck in the argument appears to be getting good control on the exponential sum.
 The exponent pair hypothesis 
$$ \mbox{would suggest that maybe} \qquad \left|\sum_{a=1}^{n} e^{2\pi i \ell \sqrt{a}} \right| \leq c_{\varepsilon} \cdot \ell^{\varepsilon} \cdot n^{1/2}.$$
It is not clear to us whether, given the completely explicit nature of the exponential sum, it is possible to obtain such a result (see \S 2.4).
Such a bound would imply the existence of examples with 
$$0 < \| \sqrt{a_1} + \dots + \sqrt{a_k}\| \lesssim n^{-k/2 + \varepsilon}.$$
There is some indication that this might be a natural limit of this very simple approach. Even simpler counting arguments suggest that one might perhaps except a rate of at least $\lesssim n^{1/2 - k}$. Given the large number of existing area, one might be hopeful that further improvements are possible.

\section{Proof}

\subsection{Proof of the Theorem.} A natural setting for the problem in terms of Fourier Analysis is to reduce everything mod 1 and work on the one-dimensional torus. We use the function $h:\mathbb{T} \rightarrow \mathbb{R}$
$$ h(x) = \max \left\{1 - s \cdot |x|, 0 \right\},$$
where $s \gg 1$ is a large parameter that will then be suitable chosen. It may be helpful to think of $s \sim n^A$ with $A \sim  c \cdot k^{1/3} \ll k/2$ which is the scale it will end up being. For $\ell \neq 0$, the Fourier coefficients in
$$ h(x) = \sum_{\ell \in \mathbb{Z}} \widehat{h}(\ell) e^{2 \pi i \ell x}$$
have a simple closed-form expression: $\widehat{h}(0) = 1/s$ and, for $\ell \neq 0$,
$$ \widehat{h}(\ell) = \frac{1}{\pi^2} \frac{ s}{\ell^2}\sin{\left(\frac{\ell \pi}{s}\right)}^2.$$ 
The precise form of the Fourier coefficients will not be important. They are non-negative and a convenient upper bound for them is
 $$ \widehat{h}(\ell) \lesssim  \begin{cases} 1/s \qquad &\mbox{if}~\ell \lesssim s \\
s/\ell^2 \qquad &\mbox{if}~\ell \gtrsim s. \end{cases}$$ 
Plugging in and exchanging the order of summation
\begin{align*}
\# \left\{(a_1, \dots, a_k) : \| \sqrt{a_1} + \dots + \sqrt{a_k}\| \leq 1/s \right\} &\geq \sum_{a_1,\dots, a_k = 1}^{n} h(\sqrt{a_1} + \dots + \sqrt{a_k} ) \\
 &=  \sum_{a_1,\dots, a_k = 1}^{n}\sum_{\ell \in \mathbb{Z}} \widehat{h}(\ell) e^{2 \pi i \ell (\sqrt{a_1} + \dots + \sqrt{a_k})} \\
 &= \sum_{\ell \in \mathbb{Z}} \widehat{h}(\ell) \left(\sum_{a=1}^{n} e^{2\pi i \ell \sqrt{a}}\right)^k.
 \end{align*}
The frequency $\ell = 0$ leads to a sizeable contribution since
$$ \widehat{h}(0) = \frac{1}{s} \qquad \mbox{and} \qquad  \left(\sum_{a=1}^{n} e^{2\pi i 0 \sqrt{a}}\right)^k = n^k.$$
Thus
$$ \sum_{\ell \in \mathbb{Z}} \widehat{h}(\ell) \left(\sum_{a=1}^{n} e^{2\pi i \ell \sqrt{a}}\right)^k = \frac{n^k}{s} + \sum_{\ell \in \mathbb{Z}\atop \ell \neq 0} \widehat{h}(\ell) \left(\sum_{a=1}^{n} e^{2\pi i \ell \sqrt{a}}\right)^k.$$
Since $\widehat{h}(\ell) = \widehat{h}(-\ell)$, we can deduce
 $$ \sum_{\ell \in \mathbb{Z}} \widehat{h}(\ell) \left(\sum_{a=1}^{n} e^{2\pi i \ell \sqrt{a}}\right)^k \geq \frac{n^k}{s} - 2 \sum_{\ell =1}^{\infty} \widehat{h}(\ell) \left|\sum_{a=1}^{n} e^{2\pi i \ell \sqrt{a}}\right|^k.$$
  The trivial solutions, where all the $a_i$ are square numbers themselves, contribute 
 $$ \sum_{a_1,a_2,\dots, a_k = 1}^{n} h(\sqrt{a_1} + \sqrt{a_2} + \dots + \sqrt{a_k} ) =\left\lfloor \sqrt{n} \right\rfloor^k \qquad \mbox{trivial solutions}$$
 that we are not interested in. Our goal is to show, for a suitable choice of $s \ll n^{k/2}$, 
 $$  \sum_{\ell \in \mathbb{Z}\atop \ell \neq 0} \widehat{h}(\ell) \left(\sum_{a=1}^{n} e^{2\pi i \ell \sqrt{a}}\right)^k \ll \frac{n^k}{s}$$
  which then implies the existence of many ($\gtrsim n^k/s \gg n^{k/2}$) solutions of $ \| \sqrt{a_1} + \dots + \sqrt{a_k} \| \lesssim 1/s$.
  Since this number exceeds the number of trivial solutions, there has to be one that is not an integer. We note that this is one of the cases in the argument where one might hope for some improvement: it is perfectly conceivable that many of the exponential sums have an actual positive contribution (whereas our argument, bounding them by their absolute value, assumes that they are always very negative). In particular, since there are always $n^{k/2}$ trivial solutions, independently of how large $s$ is, we can deduce that for $s \gg n^{k/2}$ the net contribution of the exponential sums has to be positive. However, this seems to be harder to quantify than bounds on the absolute value.\\

\textbf{Setup.} We can now describe the overall setup.  Since we leave the constant $c>0$ in the exponent $-c \cdot k^{1/3}$ unspecified 
and always have the trivial exponent $-1/2$, it suffices to prove the inequality for $k \geq k_0$ sufficiently large (one part of our argument requires $k \geq 240$). We will think of $k$ as a large fixed number and of the scale $s \sim n^A$ as a polynomial quantity where $A$ is to be determined. The goal of the proof is to derive suitable conditions on $k$ and $A$ under which we have
$$   \sum_{\ell =1}^{\infty} \widehat{h}(\ell) \left|\sum_{a=1}^{n} e^{2\pi i \ell \sqrt{a}}\right|^k \ll \frac{n^k}{s}$$
which then implies the result at scale $1/s \sim n^{-A}$. In particular, we would like $A$ to be as large as possible.
We start by splitting the sum as
\begin{align*}
\sum_{\ell =1}^{\infty} \widehat{h}(\ell) \left|\sum_{a=1}^{n} e^{2\pi i \ell \sqrt{a}}\right|^k &\leq \sum_{1 \leq \ell \leq s^{2+\varepsilon}}^{\infty} \widehat{h}(\ell) \left|\sum_{a=1}^{n} e^{2\pi i \ell \sqrt{a}}\right|^k \\
&+  \sum_{\ell \geq  s^{2+\varepsilon}}^{\infty} \widehat{h}(\ell) \left|\sum_{a=1}^{n} e^{2\pi i \ell \sqrt{a}}\right|^k,
\end{align*}
where $\varepsilon > 0$ is any arbitrarily small number (whose precise value will turn out to not be very important).
The second sum turns out to be harmless since, using the Taylor expansion together with the trivial estimate $\leq n$ on the exponential sum gives
\begin{align*}
 \sum_{\ell \geq  s^{2+\varepsilon}}^{\infty} \widehat{h}(\ell) \left|\sum_{a=1}^{n} e^{2\pi i \ell \sqrt{a}}\right|^k  &\lesssim  \sum_{\ell \geq  s^{2+\varepsilon}}^{\infty} \frac{s}{\ell^2}  n^k \lesssim \frac{n^k}{s^{1+\varepsilon}} \ll \frac{n^k}{s}.
\end{align*}
It thus remains to analyze the behavior of the first sum. Again appealing to the unspecified nature of the constant $c>0$, we can additionally assume that $A \geq 4$ (this is not really crucial for the argument but simplifies exposition). We can then split the remaining sum into summation over three different regions
\begin{align*}
  \sum_{1 \leq \ell \leq s^{2+\varepsilon}}^{} \widehat{h}(\ell) \left|\sum_{a=1}^{n} e^{2\pi i \ell \sqrt{a}}\right|^k &\lesssim  \sum_{1 \leq \ell \leq n^4}^{} \frac{1}{s} \left|\sum_{a=1}^{n} e^{2\pi i \ell \sqrt{a}}\right|^k + \sum_{n^4 \leq \ell \leq s}^{} \frac{1}{s} \left|\sum_{a=1}^{n} e^{2\pi i \ell \sqrt{a}}\right|^k\\
  &+ \sum_{s \leq \ell \leq s^{2+\varepsilon}}^{} \frac{s}{\ell^2} \left|\sum_{a=1}^{n} e^{2\pi i \ell \sqrt{a}}\right|^k.
\end{align*}
 The remainder of the argument is as follows: the first sum is benign in the sense that we can use van der Corput estimates to show it to be $\ll n^k/s$ for $k > 240$. The remaining two sums are only split because of the different algebraic structure of the coefficient ($1/s$ vs $s/\ell^2$). We will use the same approach to treat the exponential sum in that regime. Iterated derivatives of the phase function are sufficiently large for Vinogradov's method to apply which leads to a polynomial gain.\\

\textbf{van der Corput estimates.}
We use the van der Corput Method to prove that
$$ \max_{1 \leq \ell \leq n^4} \left|\sum_{a=1}^{n} e^{2\pi i \ell \sqrt{a}}\right| \lesssim n^{59/60}.$$
For all $k > 240$, we can then deduce
$$ \sum_{1 \leq \ell \leq n^{4}}^{} \frac{1}{s} \left|\sum_{a=1}^{n} e^{2\pi i \ell \sqrt{a}}\right|^k \leq  \frac{n^4}{s} n^{\frac{59k}{60}} \ll \frac{n^k}{s}$$
which shows the sum to be small compared to the main term.
We quickly sketch the relevant details. One formulation of the van der Corput method is as follows (taken from Tenenbaum \cite{tenenbaum}): if $f:[m, 2m] \rightarrow \mathbb{R}$ is $R$ times differentiable and, for all $1 \leq r \leq R$ and all $m \leq x \leq 2m$ one has
$$ F \cdot m^{-r} \lesssim |f^{(r)}(x)| \lesssim F \cdot m^{-r},$$
then one obtains
$$ \left| \sum_{m \leq x \leq 2m} e^{2 \pi i f(x)} \right| \lesssim \left( F^u m^{-v} + F^{-1} \right)m,$$
where $u = 1/(2^R-2)$ and $v = R/(2^R-2)$. In our setting, we have $f(x) = \ell \sqrt{x}$ and therefore $F = \ell \sqrt{m}$ where the constants are absolute (there is an implicit constant depending on the number of the derivatives that grows in $r$ but since we will be able to set $R=5$, these constants can be taken to be absolute). Note that $F^{-1}m \lesssim m^{1/2}$, so this term is completely harmless. It remains to deal with the term $F^u m^{-v} m$.
Since, by assumption, $\ell \leq n^4$, we get the estimate
\begin{align*}
 \left| \sum_{m \leq x \leq 2m} e^{2 \pi i \ell \sqrt{x}} \right| &\lesssim F^u m^{-v} m \leq (n^4 \sqrt{m})^u  m^{-v} m \lesssim n^{4 u} m^{1 + u/2 - v}.
 \end{align*}
Summation over all dyadic scales gives
$$ \left| \sum_{1 \leq x \leq n} e^{2 \pi i \ell \sqrt{x}} \right| \lesssim n^{4 u} n^{1 + u/2 - v} = n^{1 + 9u/2 - v}.$$
It thus remains to choose $R$ in such a way that $9u/2 - v < 0$ which leads to the choice $R = 5$. We note that the argument requires $k > 240$, however, since we do not control the implicit constant $c$ in the main result, this is sufficient.\\

\textbf{Vinogradov method.}
It remains to analyze the other two remaining sums, those ranging over $n^4 \leq \ell \leq s$ and $s \leq \ell \leq s^{2+\varepsilon}$.
We use the Vinogradov method in the formulation given in the book by Iwaniec-Kowalski \cite{iwaniec}.
\begin{thm}[Vinogradov  \cite{iwaniec}] Let $f(x)$ be a smooth function on $[m, 2m]$ satisfying, for all $j \geq 1$ and all $m \leq x \leq 2m$ that
$$ \alpha^{-j^3} F \leq \frac{x^j}{j!} \cdot \left| f^{(j)}(x)\right| \leq \alpha^{j^3} F,$$
where $F \geq m^4 \geq 2$ and $\alpha \geq 1$. Then
$$ \left| \sum_{m \leq x \leq 2m} e(f(x)) \right| \lesssim \alpha \cdot m \cdot \exp\left( - 2^{-18} \frac{(\log{m})^3}{(\log F)^2}\right)$$
and the implied constant is absolute.
\end{thm}
We are dealing with the special case is $f(x) = \ell \sqrt{x}$. There, 
$$ \frac{x^j}{j!} \cdot \left| f^{(j)}(x)\right| = c_j \cdot \ell \cdot \sqrt{x},$$
where $|c_j| \leq 1/2^j$. This allows us to set $F = \ell \sqrt{n}$ and $\alpha = 2$. Since $\ell \geq n^4$, we have
$ F = \ell \sqrt{n} \geq n^{4.5} \geq n^4 \geq m^4$ and thus the condition $F \geq m^4$ is always satisfied. Moreover, since we only care about the frequencies $n^4 \leq \ell \leq s^{2+\varepsilon} = n^{2A + \varepsilon}$, we have
$$ F  = \ell \sqrt{m} \leq n^{2A + \varepsilon} n^{1/2} \leq n^{2A + 1}$$
and thus $(\log F)^2 \lesssim (2A + 1) \log{n}$. As a consequence, we have
$$ \exp\left( - 2^{-18} \frac{(\log{m})^3}{(\log F)^2}\right) = m^{ - 2^{-18} \frac{(\log{m})^2}{(\log F)^2}} \lesssim m^{ -\frac{c}{A^2} \frac{(\log m)^2}{(\log{n})^2}}.$$
Dyadic summation leads to
\begin{align*}
 \left| \sum_{1 \leq x \leq n} e^{2 \pi i \ell \sqrt{x}} \right| &\leq \sum_{1 \leq r \leq \log{n}} \left|  \sum_{2^r \leq x \leq 2^{r+1}} e^{2 \pi i \ell \sqrt{x}}  \right| \lesssim  \sum_{1 \leq r \leq \log{n}} 2^{r\left( 1 - \frac{c}{A^2} \frac{r^2}{(\log n)^2}\right)}.
\end{align*}
We distinguish two cases: if $r \leq (\log n)/2$, then the sum is trivially $\lesssim \sqrt{n}$ which is tiny. If $(\log{n})/2 \leq r \leq \log{n}$, we deduce
$$  \left| \sum_{1 \leq x \leq n} e^{2 \pi i \ell \sqrt{x}} \right| \lesssim  \sum_{ (\log{n})/2 \leq r \leq \log{n}} 2^{r\left( 1 - \frac{c/4}{A^2} \right)}\lesssim n^{1 - c'/A^{2}}.$$
Note that this estimate is uniform in $n^4 \leq \ell \leq n^{2A+2}$.
This is enough to bound the two remaining sums. For the first part, we may simply consider the maximum
\begin{align*}
   \sum_{n^4 \leq \ell \leq s}^{} \frac{1}{s} \left|\sum_{a=1}^{n} e^{2\pi i \ell \sqrt{a}}\right|^k \lesssim \max_{n^4 \leq \ell \leq n^A}\left|\sum_{a=1}^{n} e^{2\pi i \ell \sqrt{a}}\right|^k  \lesssim n^{(1-c'/A^2)k}.
\end{align*}
As for the remaining sum, we can use
$$ \sum_{s \leq \ell \leq s^{2+\varepsilon}} \frac{s}{\ell^2} \leq  \sum_{\ell \geq s} \frac{s}{\ell^2} \sim 1$$
to
argue similarly and obtain again
\begin{align*}
 \sum_{s \leq \ell \leq s^{2+\varepsilon}}^{} \frac{s}{\ell^2} \left|\sum_{a=1}^{n} e^{2\pi i \ell \sqrt{a}}\right|^k \lesssim \max_{s \leq \ell \leq n^{(2+\varepsilon)A}}\left|\sum_{a=1}^{n} e^{2\pi i \ell \sqrt{a}}\right|^k  \lesssim n^{(1-c'/A^2)k}.
\end{align*}

It remains to understand when
$$ n^{(1- cA^{-2} )k} \ll \frac{n^k}{s} = n^{k-A}$$
which needs $k \geq A^3/c$. Conversely, for $k$ sufficiently large, we may choose $A$ as large as $A \leq c \cdot k^{1/3}$ which implies the result. $\qed$

\subsection{Concluding Remarks.}
One of the advantages is that the exponential sum $\sum_{1 \leq a \leq n} e^{2 \pi i \ell \sqrt{a}}$ is completely explicit. One way one could hope to avoid it is to relate it to the continuous integr to relate
$$ \sum_{1 \leq a \leq n} e^{2 \pi i \ell \sqrt{a}}  \qquad \mbox{and the integral} \qquad  \int_1^n e^{2 \pi i \ell \sqrt{x} - \nu x } dx,$$
where we refer to \cite[Section I.6, Theorem 6.4]{tenenbaum} for details. There exists an explicit closed-form expression for the antiderivative $X= \int_1^n e^{2 \pi i \ell \sqrt{x} - \nu x } dx$
$$X =\frac{i \left(\sqrt[4]{-1} \sqrt{2} \pi  \ell  e^{\frac{i \pi  \ell ^2}{2 \nu }}
   \text{erfi}\left(\frac{\left(\frac{1}{2}+\frac{i}{2}\right) \sqrt{\pi } \left(\ell -2 \nu  \sqrt{x}\right)}{\sqrt{\nu }}\right)+2
   \sqrt{\nu } e^{-2 i \pi  \left(\nu  x-\ell  \sqrt{x}\right)}\right)}{4 \pi  \nu ^{3/2}}\big|_1^n $$
The presence of a completely explicit term that induces further oscillations does seem to inspire some hope that are more precise understanding might be possible, however, it does not appear to be straightforward.\\

\textbf{Acknowledgment.} The author is grateful for interesting discussions with Roger Heath-Brown, Arturas Dubickas, Charles Greathouse, Nicholas Marshall, Tom\'as Oliveira e Silva and Victor Reis. I am especially indebted to Igor Shparlinski for insightful remarks and finding a gap in an earlier version of the manuscript.

\end{document}